\documentclass[arxiv,reqno,twoside,a4paper,12pt]{amsart}
\usepackage{tikz}
\usepackage{amsmath, verbatim}
\usepackage{amssymb,amsfonts,mathrsfs,mathtools}
\usepackage[colorlinks=true,linkcolor=blue,citecolor=blue]{hyperref}
\usepackage[driver=pdftex,heightrounded=true,centering]{geometry}
\usepackage{longtable, slashed}
\usepackage{tabularx}
\usepackage{multirow}
\usepackage{caption}
\usepackage{latexsym,enumitem}
\usepackage[all]{xy}

\usepackage[latin2]{inputenc}
\usepackage[mathscr]{eucal}
\usepackage{t1enc, indentfirst}
\usepackage{graphicx, graphics}
\usepackage{pict2e, epic, amssymb }
\usepackage{tikz, pst-node}
\usepackage{tikz-cd, pgfplots}
\usetikzlibrary{arrows}
\usetikzlibrary{calc, patterns}
\usepackage{textcomp}
\pgfplotsset{compat=1.9}
\usepackage{epstopdf}
\usepackage{xcolor}

\usepackage[scaled]{helvet} 
\usepackage{courier} 
\usepackage[mathbf]{euler}

\theoremstyle{plain}
\newtheorem{Th}{Theorem}[section]
\newtheorem{Lemma}[Th]{Lemma}

 \theoremstyle{definition}
\newtheorem{Def}[Th]{Definition}
\newtheorem{defn}[Th]{Definition}

\newtheorem{Rem}[Th]{Remark}
\newtheorem{?}[Th]{Problem}
\newtheorem{Ex}[Th]{Example}
\newtheorem{Prob}[Th]{Open problem}

\numberwithin{equation}{section}

\definecolor{qqwuqq}{rgb}{0,0,0}

\begin{document}

\title[Analytic torsion]
{Analytic torsion  on manifolds with fibred boundary metrics}

\author{Mohammad Talebi}
	\address{Universit\"at Oldenburg, Germany}
	\email{m0hammadtalebi@aol.com}
	\date{\today}

\maketitle

\begin{abstract} In this paper, we construct the renormalized
analytic torsion in the setup
of manifold endowed with fibred boundary  metrics.
The method of construction is to determine the 
asymptotic of heat kernel, both in short time regime
and long time regime and apply these asymptotics together with renormalization to
determine the  renormalized zeta function and the determinant of
Hodge Laplacian.
\end{abstract}

\section{Introduction} \label{sec intro}
Analytic torsion was introduced by Ray and Singer \cite{raysin}
as analytic 
counterpart of Reidemeister torsion in topology.
 Ray and Singer conjectured that these two torsions 
 are equivalent on closed manifolds. Cheeger and M\"uller
  proved this conjecture independently.
Assume that $(M,g)$ is a closed Riemannian manifold and 
$e^{-t\Delta_{g}}(x,y) := H(t,x,y)$ is the heat kernel with respect to Hodge Laplacian $\Delta_{g}^{q}: \Omega^{q}(M) \longrightarrow \Omega^{q}(M)$, acting on the space of $q$ forms.
 i.e  fundamental solution to heat 
equation,
\begin{align} \label{heat eq}
&\partial_{t}H(t,x,y) + \Delta_{g,x}^{q}H(t,x,y) = 0,\\ \nonumber
&H(t = 0, x ,y ) = \delta (x-y).
\end{align}
One define the heat trace to be, $Tr(e^{-t\Delta_{g}^{q}}) = \int_{M} e^{-t\Delta_{g}}
(x,x)dvol_{g}.$
Assume $\Delta_{g}^{q}$ is Hodge Laplacian acting on the space of
 $q$-forms.
The corresponding  zeta function is defined to be,
\begin{equation}\label{zetaclassic}
\zeta_{q}^{M}(s) : = \frac{1}{\Gamma(s)}\int_{M}Tr(e^{-t\Delta_{g}^{q}})t^{s-1}dt.
\end{equation}
So defined (\ref{zetaclassic}) is defined on 
$Re(s) > \frac{n}{2}$ where $n = dim(M)$  but can be extended holomorphically to complex
plane $\mathbb{C}$  and especially with regular point at $s =0$ . 
The determinant of Laplacian is defined to be,
\begin{equation*}\label{detlap}
det(\Delta_{g}^{q}) := e^{-\zeta'_{q}(0)},
\end{equation*}
One defines the analytic torsion as,
\begin{equation*}\label{analtorclas}
\log T(M) := \frac{1}{2}\sum_{q =0}^{n}(-1)^{q}q\zeta'_{q}(0),
\end{equation*}
typically one would like  to generalize the definition 
of zeta function and analytic torsion on non 
compact set up or non smooth set up and obtain 
Cheeger M\"uller type statements on the new set up.
 we refer to the PhD thesis of Sher \cite{She-stab}
  in which he constructed the zeta function in the
   set up of asymptotic conic manifolds.

\medskip
In this work we consider manifold $\overline{M}$ with boundary $\partial M$ with fibration boundary structure, i.e $\partial M$ is fibred over closed base $B$ with typical closed fibre $F$.
On such a topological structure one may consider metric
  $g_{\phi}$ defined as,
\begin{equation*}
g_{\phi} = \frac{dx^{2}}{x^{4}} + \frac{\phi^{*}g_{B}}{x^{2}} + g_{F} + h(x),
\end{equation*}
where $x$ is a boundary defining function of $\partial M$,
 and $\phi$ is fibration and $g_{B}$ is a Riemannian metric 
 over base $B$ and $g_{F}$ is a symmetric bilinear form which restricts to Riemannian metric on fibre $F$ and additional assumptions as in \cite{reslaw gtv}. 
In this set up, we are going to define the concept of analytic torsion. The main difficulty arises when we consider the heat trace,
 \begin{equation*}
Tr(e^{-t\Delta_{g_{\phi}}^{q}}) = \int_{\mathbb{R}^+} e^{-t\Delta_{g_{\phi}}}
(x,x)dvol_{\phi},
\end{equation*}
i.e in $\phi$ set up, the boundary is located at infinity and therefore the integration over diagonal diverges. To address this problem, the renormalized heat trace is described in the hadamard manner, which essentially
takes into account the integration of the heat kernel along the diagonal on the finite component. The heat kernel structure theorem may be employed to take the finite part of this integral at zero
to be heat trace renormalized . 
In order to describe analytic torsion in the set up of
$\phi$ manifolds, we can explicitly describe renormalized zeta function and Laplacian determinant by means of renormalized heat trace.\\

The paper is organized as follows.
In section \ref{sec 2} we recall the set up of $\phi$
 manifolds and 
some definitions from geometric microlocal analysis
 in the sense of Melrose \cite{MelATP}, which we need
 in order to describe structure theorems of heat kernel in section \ref{sec 3}.
 We remark that these methods are in effect microlocal
 analysis due to H\"ormander, Nierenberg and Maslow
 together with manifold with corners due to Cherrf and Duady \cite{mancor} and the type of resolution process 
 basically arising in the algebraic geometric sense \cite{res} of resolution of singularities. Melrose uses these methods in order to develop elliptic theory of differential operators in several geometric set ups, as 
  b geometry \cite{MelATP} and $\phi$ geometry \cite{maz-pseudo}.\\
In section \ref{sec 3} we describe structure theorems
 of heat kernel both in 
finite time regime and long time regime on appropriate space.
 Namely on  manifolds with corners. 
In finite time regime this description is explicit \cite{heat short phi} and in long time regime we use functional calculus and asymptotic of $\phi$ resolvent of $\phi$
Hodge Laplacian at low energy \cite{reslaw gtv}
to relate heat kernel and resolvent via,
\begin{equation*}
H^{M}(t,z,z') = \int_{\Gamma} 
e^{-t\lambda}(\Delta_{\phi}-\lambda)^{-1}d\lambda,
\end{equation*}
where $\Gamma$ is a curve around the spectrum of Hodge Laplacian $\Delta_{\phi}$ oriented counter-clockwise.   \\
In section \ref{sec4} we use these structure theorems in order to determine the asymptotics of renormalized heat trace and show that the renormalized zeta function admits meromorphic continuation on whole of complex plane $\mathbb{C}$. One defines then the renormalized determinant of Laplacian and analytic torsion in the set up of $\phi$ manifolds.\\
This work is generalization of \cite{She-stab} where
 the fibres are trivial. Namely we generalize the result of 
  \cite{She-stab} for closed fibre $F$.\\
\medskip

\textit{Aknowledgment:} The author would like to thank Boris Vertman 
for his supervision for PhD thesis. Further he acknowledges the constructive comments of Daniel Grieser, Julie Rowlett and Boris Vertman  on improvements and corrections of this work.

\section{phi manifolds and geometric microlocal analysis}\label{sec 2}
Assume $\overline{M}$ is a compact manifold with boundary 
$\partial M$ and $\partial M$ has fibration structure i.e,
 $\partial M \overset{\phi}{-} B - F,$
where $\phi$ is trivialization of fibration. $B$ is base manifold and $F$ is closed manifold as fibre. Near the boundary 
one may take the product $[0,\epsilon)\times \partial M$ by
 collar neighborhood theorem, and
fix local coordinates on 
$\overline{M}$ to be, $
(x, y = (y_{1}\cdots y_{b}), z = (z_{1},\cdots z_{f}))$.
Here $ x = \rho_{\partial M}$ is  the boundary
 defining function of 
$\partial M$, i.e, $
\partial M = \{ x= 0\},\, dx \neq 0,\, x \geq 0.$
Consider the metric,
\begin{equation*}
g_{\phi} = \frac{dx^{2}}{x^{4}} + \frac{\phi^{*}g_{B}}{dx^{2}}
 + g_{F},
\end{equation*}
on $\overline{M}$ where $g_{B}$ is Riemannian metric on base $B$
 and 
$g_{F}$ is symmetric bilinear form which restricts to Riemannian metric on fibre $F$. We assume further that $\varphi: (\partial M,g_{F} + \varphi^{*}g_{B}) \longrightarrow (B,g_{B})$ is
 Riemannian submersion. Such a geometric set up is called fibred
boundary $\phi$ metric manifolds.
Intuitively the boundary is fibre bundle with base $B$
 and fibre $F$
where the boundary is viewed to be located at infinity.
The metric arises as example in gravitaitional 
instantons \cite{gravi ins}. Gravitaitional instantons 
are defined as $4$ dimension complete
 hyperk\"ahler manifolds.
They are classified in three categories. ALE or 
asymptotically locally euclidean where fibre $F$ is  a point.
The classification of ALE manifolds are given by Kronheimer
in his P.h.D thesis \cite{ale kron} and the
underlying manifold is 
minimal resolution of
$\frac{\mathbb{C}}{\Gamma}$ where $\Gamma$ is a finite 
subgroup of $SU(2)$ of type $A_{k}, D_{k}, E_{6}, E_{7}, E_{8}$.\\
The next type of gravitaitional instantons
 are ALF or asymptotically locally flat manifold 
 where fibre $F$ is  $\mathbb{S}^{1}$ and the last
  types are ALG where the fibre is
   $ F = \mathbb{S}^{1}\times\cdots \times\mathbb{S}^{1}$. We present now two simple examples of fibred boundary $\phi$ manifolds.

\begin{Ex}\label{ example 1}
Consider Euclidean space $\mathbb{R}^{n}$ and the Euclidean metric
given in polar coordinates $(r,\theta), r \in \mathbb{R}, \theta \in \mathbb{S}^{n-1}$, $g_{euc} = dr^{2} + r^{2}d\theta^{2}$. One may 
compactify Euclidean space by introducing change of variable,
$ x =\frac{1}{r}$. The metric $g_{euc}$ takes then the
 form, $
g_{sc} = \frac{dx^{2}}{x^{4}} + \frac{d\theta}{x^{2}},$
which is a simple example of $\phi$ metric with
 trivial fibre $F = \{point\}$ and base $B = \mathbb{S}^{n-1}$.
\end{Ex}

\begin{Ex}
Consider now  $\mathbb{R}^{n}\times F$ where
 $F$ is closed Riemannian manifold $(F,g_{F})$.
  Use the product metric $g_{euc} \oplus g_{F}$
and introduce change of variable $x = \frac{1}{r}$
 in the euclidean metric 
$dr^{2} + r^{2} d\theta^{2}$ as \ref{ example 1}. 
One gets the metric,  
$g = \frac{dx^{2}}{x^{4}} + \frac{d\theta}{x^{2}} + g_{F},$
which is again  example of $\phi$ metric with
fibre $F $ and base $B = \mathbb{S}^{n-1}$ and product
 $\mathbb{S}^{n+1}\times F$ at boundary.
\end{Ex}
After introducing some examples of $\phi$ metrics,
one may ask  classical analytic or geometric 
questions on such a setup.
 As such a question,
 consider heat 
equation \ref{heat eq} and fundamental solution to it
 which we denote it as $e^{-t\Delta_{g_{\phi}}}(x,x')$
  by suppressing $y,z ,y',z'$.
 In order to define the analytic torsion for $\phi$ manifolds, one needs to determine asymptotic of heat kernel  in finite time and long time regimes . To that aim, one construct resolution space which are manifold with corners and lift the integral kernel of heat kernel to those space to obtain polyhomogeneous conormal distributions in the sense made precise below. We recall the required definitions which are needed in the section \ref{sec 3} in order to study heat kernel structure. Our main references are \cite{basics} 
and \cite{MelATP} and \cite{Mel-diff}.
We define first the model space.

\begin{Def}(Manifold with corners) One defines,
\begin{itemize}
\item A manifold with corners $M$ of dimension $n$ is
 a Hausdorf 
topological space which is locally diffeomorph to 
$\mathbb{R}^{k}_{+}\times\mathbb{R}^{n-k}$.\\
\item For a manifold with corners $M$, the interior
 $\overset{o}{M}$ is the set of interior points of $M$
  i.e those points with neighborhood diffeomorph to
   $\mathbb{R}^{n}$. The boundary is 
$M \backslash \overset{o}{M}$ and is defined as union 
of connected
boundary hypersurfaces. Any intersection of two 
or more boundary hypersurfaces is called a corner.
\item For a boundary hypersurface $H$ one denote 
$\rho$ as boundary defining function for $H$ if $\rho$ is smooth in interior of $M$ and 
$\rho = 0$ on $H$ and $d\rho \neq 0$ on $H$.
\item $p$ submanifold of manifold with corner is 
defined to be 
locally coordinate submanifold.
\end{itemize}
\end{Def}
\begin{Def}(Blow up and coordinates)\label{blow up}
Assume $ P \subset M$ is a p submanifold of manifold 
with corners
$M$. For each $ p \in P$ replacing $p$ by inward spherical
 normal space yields to a space which we call it blown up
  space and it is denoted by $[M ; P]$. The process is 
  called blow up and the map,
\begin{align*}
&\beta: [ M; P] \longrightarrow M,\\
&\beta\restriction_{M \backslash \{P\}} := \text{Id},\,\beta\restriction_{P} := \mathbb{S}N(p)\mapsto p. \quad \forall \, p \in P.
\end{align*}
is called blow down map.\\
One may introduce coordinates in order to do analyse on blown
up space. 
The coordinates vary from polar coordinate on normal 
sphere or projective coordinates or as example logarithmic
 coordinates depending on the nature of problem that we deal with.
\end{Def}
On manifold with corners, the main object of study
 are distributions.
In the next definition we characterise
 polyhomogeneous distributions.
 \begin{Def}(Polyhomogeneous function)\label{polyho}
Assume $M$ is a manifold with corners and 
$\mathcal{E} = \mathcal{E}\{(E_{i}, H_{i})\}$
 is a index family as explained in \cite{basics}. 
A function $u$ defined on interior of $M$
 is polyhomogeneous conormal
on $M$ with index family $\mathcal{E}$
 and is denoted as,
$ u \in \mathcal{A}_{phg}^{\mathcal{E}}(M)$
 if in any neighborhood coordinates of boundary point $p$ ,$
u(x_{1},\cdots x_{n}) \sim \sum_{\{(z_{i},p_{i}\}\in E_{i}} a_{(z_{i},p_{i})}\prod_{i=1}^{k}x_{i}^{z_{i}}(log x_{i})^{p_{i}},$
where $a_{z_{i},p_{i}}$ are polyhomogeneous with respect to index set obtained from intersection of boundary hypersurfaces.
\end{Def}
We define now interior conormal singularity.
\begin{Def} \label{conormal def}
Assume $P \subset M$ is a interior p-submanifold.
 A distribution $u$ on $P$ is
conormal at $P$ of order $m$ if it is smooth away from $P$ an near $P = \{z=0\}$ can be expressed locally as,
\begin{equation*}
u(y,z) = \int_{\mathbb{R}^{n-k}} e^{iz\cdot\zeta}a(y, \zeta)d\zeta,
\end{equation*}
where $a$ is classical symbol of order $m$.
Order $m$ means that, $$
a(y, \zeta) \sim \sum_{j=0}^{\infty}a_{m-j}(y, \frac{\zeta}{\vert \zeta \vert})\vert \zeta \vert^{m-j},$$
with each coefficient $a_{m-j}$ smooth in $y$ and $\frac{\zeta}{\vert \zeta \vert}$. 
If $P$ intersect the boundary one require $a_{m-j}$ to be polyhomogeneous.

\end{Def}

The residue theorem may be applied in order to relate the
heat kernel to the resolvent kernel.
 We recall residue theorem from \cite{Ahlfors}.

\begin{Th}(Residue theorem) \label{res}
Let $f(z)$ be a complex valued function which is
 holomorphic except 
for isolated singularities $\alpha_{j}$ in a region 
$\Omega \subset \mathbb{C}$. Then,
\begin{equation*}
\frac{1}{2\pi i}\int_{\gamma} f(z)dz = \sum_{j} n(\gamma, \alpha_{j})\text{Res}_{z = \alpha_{j}} f(z),
\end{equation*}
where $\gamma$ is a closed path  homologous to zero
in $\Omega$ and $\gamma$ does not pass through any of the points 
$\alpha_{j}$ and $n(\gamma, \alpha_{j})$ is winding number of
$\gamma$ around $\alpha_{j}$.
\end{Th}

We summarized all instruments which we need in 
order to study heat kernel on $\phi$ manifolds 
and determine the asymptotic of renormalized 
heat trace. In the next section \ref{sec 3}, 
we study the structure theorems of heat kernel
 in short time and long time regimes.

\section{Heat kernel in short time regime 
and long time regime}\label{sec 3}

We start this section by defining the heat kernel on
a compact Riemannian manifold
and continue to determine the heat kernel asymptotic 
on the manifold with  fibred boundary, $\phi$ metric
in finite time and long time. Using these asymptotics 
analytic torsion with respect to $\phi$ Hodge Laplacian  will
be constructed in section \ref{sec4}.

\begin{Def}
Assume $(M,g)$ is a closed Riemannian manifold and 
$\Delta_{g}^{q}: \Omega^{q}(M) \longrightarrow \Omega^{q}(M)$ is
 the Hodge Laplacian acting on the space of $q$ forms.
 A heat kernel is a function, 
 $H_{g,q}: M\times M \times [0,\infty) \longrightarrow M$
that satisfies,
\begin{itemize}
\item $H_{g,q}(x,y,t)$ is $C^{1}$ in $t$ and $C^{2}$ in $(x,y)$.
\item $\frac{\partial H}{\partial t} + \Delta_{g,x}^{q}(H) = 0$, where 
$\Delta_{g,x}^{q}$ is Hodge Laplacian on manifold $(M,g)$.
\item $\lim_{t\longrightarrow 0^{+}} H_{g,q}(x,y,t) = \delta(x-y)$
where by $\delta(x-y)$ we mean delta distribution.
\end{itemize}
\end{Def}
Then the solution of heat equation with initial condition $u(x,0) = f(x)$ is given by, $u(x,t) = \int_{M}H(x,y,t)f(y)dy.$\\
Recall that for $(M,g)$ smooth closed  manifold and
$\{\lambda_{j}\}$ spectrum of $\Delta_{g}$ on $M$ and $\psi_{j}$
 the associated 
eigenfunctions one may write,\\
\begin{equation*}
H(x,y,t) = \sum_{i} e^{-\lambda_{j}t}\psi_{i}(x)\psi_{i}(y),
\end{equation*}
and the heat trace is defined as,$
TrH(t) = \int_{M} H(x,x,t)dx = \sum_{i} e^{-\lambda_{j}t}.$
The heat trace has a asymptotic expansion as
 $t\longrightarrow 0^{+}$,
$$
TrH(t) \sim_{t\longrightarrow 0^+} (4\pi t)^{\frac{dim(M)}{2}}\sum_{j = 0}^{\infty} a_{j}t^{j},$$
where $a_{j}$ are integrals over $M$ of universal homogeneous polynomials in the curvature and its covariant derivatives.

We observe that the heat kernel $H(t,x,x')$ as integral 
kernel is supported
on $M^{2}\times\mathbb{R}^{+}$. We switch now
 to manifold with fibred 
boundary 
endowed with $\phi$ metric $(M,g_{\phi})$. One determines 
asymptotics of heat kernel in short and long time regimes 
by blow up process.\\
For finite-time regime this integral kernel lifts
on so called heat space to be polyhomogeneous conormal
kernel and for long-time regime the statement is to determine
polyhomogeneous kernel of resolvent of  $\phi$ Hodge
 Laplacian 
at low energy level
and then  residue Theorem \ref{res}  applies in order to
calculate  long time heat kernel asymptotics.\\

\subsection*{Finite time regime}

In order to determine the behaviour of heat kernel in short
time regime, we observe that the heat kernel initially 
is supported on $\overline{M}^{2}\times \mathbb{R}^{+}$.
One may use resolution process \ref{blow up} to
 obtain a manifold with corners
which we denote it
as $HM_{\phi}$. On $HM_{\phi}$ the heat kernel lifts to
polyhomogeneous conormal distribution in the sense of 
definition \ref{polyho}.
One constructs  in \cite{heat short phi} the integral
kernel of heat equation which lifts to polyhomogeneous conormal
distribution
on the heat space $HM_{\phi}$. \\
The heat space is constructed from
 $\overline{M}^{2} \times \mathbb{R}^{+},$
by resolution process. Namely, one takes the
elliptic $\phi$-space in time and blow up additionally
the diagonal of $\overline{M}^{2}_{\phi}\times \mathbb{R}$
at $ t = 0$. The space $HM_{\phi}$ can be visualized as in figure
 \ref{heat-space},
where we denote $\beta_{\phi-h}$ to be the blow down map.
The main result of \cite{heat short phi} reads as,

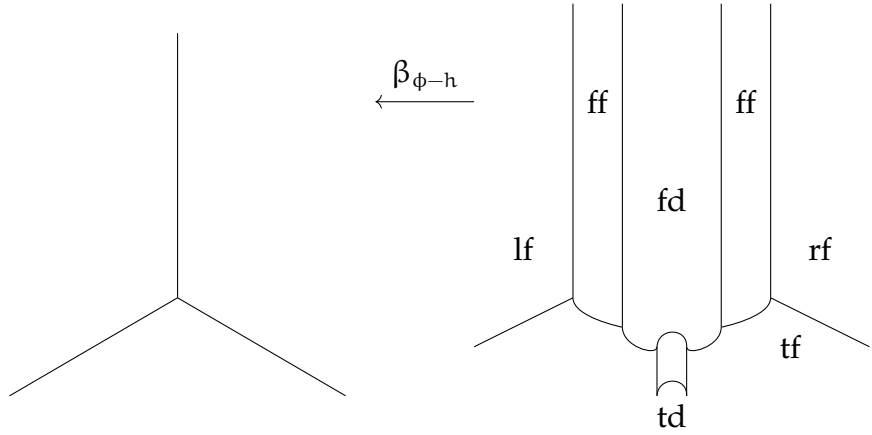
\begin{figure}[h]
\centering

\begin{tikzpicture}[scale = 1.3]
  \draw[-] (0,-1)--(0,1.7);
  \draw[-] (0,-1)--(-1.7,-2);
  \draw[-] (0,-1)--(1.7,-2);
          
    \draw[<-] (2,1) -- node[above] {$\beta_{\phi-h}$} (3,1);

   \begin{scope}[shift={(5,0)}]
     \draw[-](1,-1) --(2,-1.5);
     \draw[-](-1,-1)--(-2,-1.5);
      \draw[-](1,-1) --(1,2);
      \draw[-](-1,-1)--(-1,2);
      \draw[-](-0.5,-1.3)--(-0.5,2);
      \draw[-](0.5,-1.3)--(0.5,2);
      
          \node at (0,0) {$\text{fd}$};
         \node at (0.75,1) {$\text{ff}$};
         \node at (-0.75,1) {$\text{ff}$};
           \node at (-1.5,-0.5) {$\text{lf}$};
          \node at (1.5,-0.5) {$\text{rf}$};
           \node at (1.2,-1.5) {$\text{tf}$};
           
         \draw (1,-1).. controls (1,-1.2) and (0.5,-1.3)..(0.5,-1.3);
  
            \draw (-1,-1).. controls (-1,-1.2) and (-0.5,-1.3)..(-0.5,-1.3);        
           
       \draw[-] (0.15,-1.5)--(0.15,-2);
        \draw[-](-0.15,-1.5)--(-0.15,-2);
 \draw (-0.15,-1.5).. controls (-0.14,-1.3) and (0.16,-1.3)..(0.15,-1.5);
 \draw (-0.15,-2).. controls (-0.14,-1.8) and (0.16,-1.8)..(0.15,-2);

           \draw (0.15,-1.5)..controls (0.15,-1.6) and (0.5,-1.5)..(0.5,-1.3);
            \draw (-0.15,-1.5)..controls (-0.15,-1.6) and (-0.5,-1.5)..(-0.5,-1.3);

          \node at (0,-2.2) {$\text{td}$};
\end{scope}

\end{tikzpicture}

\caption{The heat blowup space $HM_{\phi}$.}
\label{heat-space}
\end{figure}

\begin{Th}\cite{heat short phi}(Theorem 7.2) \label{short time}
With the same assumptions as in \cite{heat short phi}, the fundamental solution of the heat equation, for finite time $t <\infty,$
\begin{align*}
&\partial_{t}H(t,x,x') + \Delta_{g_\phi,x}^{q}H(t,x,x') = 0,\\
& H(t = 0,x,x') = \delta(x-x'),
\end{align*}
lifts to polyhomogeneous conormal distribution on 
$HM_{\phi}$ with leading asymptotics $0$ at 
$fd$ and $-n$ at td and vanishing to infinite order
 on other  hypersurfaces of $HM_{\phi}$. Here  $n  = \text{dim}\,\overline{M}$.
\end{Th}

Consequently, Theorem \ref{short time} completely 
determines behaviour of heat kernel in finite time regime.

\subsection*{Long time regime}

The relation between heat kernel and resolvent
can be read from residue Theorem \ref{res}.
Namely, we may
express heat kernel as,
\begin{equation} \label{func res heat}
H^{M}_{\phi,q}(t,z,z') = \frac{1}{2\pi i}
\int_{\Gamma'} e^{-t\lambda}
(\Delta_{\phi}^{q} - \lambda)^{-1}d\lambda,
\end{equation}
Where $\Gamma$ is the curve around the spectrum oriented counter-clockwise and
 with $(\Delta_{\phi} - \lambda)^{-1}$ 
 we understand the integral kernel of resolvent
  $(\Delta_{\phi} - \lambda)^{-1}$.
Note that $\Delta_{\phi}$ admits positive 
continuous real spectrum and consequently 
one  visualize $\Gamma'$ as,
 figure \ref{counter1}.
\begin{figure}[h]
\centering

\begin{tikzpicture}[scale = 1]
  \draw[<-]  (0.5,0.5)--(2,2);
  \draw[->]  (0.5,-0.5)--(2,-2);
  \draw[dashed] (0,0) --(3,0);
  \node at (0.2,0.2) {$(0,0)$};
 \draw[thick,black] ([shift=(43:0.7cm)]0,0) arc (43:318:0.7cm);
  \end{tikzpicture}

\caption{$\Gamma'$}
\label{counter1}
\end{figure}
Or with the change of variable $\lambda = - \lambda$ in (\ref{func res heat}) one obtains,
\begin{equation} \label{func res heat 2}
H^{M}_{\phi,q}(t,z,z') = \frac{1}{2\pi i}
\int_{\Gamma} e^{t\lambda}
(\Delta_{\phi}^{q} +\lambda)^{-1}d\lambda,
\end{equation}
But now $ \Gamma$ is visualized as figure \ref{counter}.

\begin{figure}[h]
\centering

\begin{tikzpicture}[scale = 1]
  \draw[<-]  (-0.5,0.5)--(-2,2);
  \draw[->]  (-0.5,-0.5)--(-2,-2);
  \draw[dashed] (0,0) --(-3,0);
  \node at (0,0.2) {$(0,0)$};
 \draw[thick,black] ([shift=(134:0.7cm)]0,0) arc (134:-134:0.7cm);
  \end{tikzpicture}

\caption{$\Gamma$}
\label{counter}
\end{figure}

In order to explain the behaviour of heat kernel as $ t \longrightarrow \infty$ one applies (\ref{func res heat 2}).\\
In \cite{reslaw gtv} authors constructed the polyhomogeneous
integral kernel of resolvent of $\Delta_{\phi}$ on
 the space $M^{2}_{\lambda,\phi}$ at low energy level i.e as
 $\lambda \longrightarrow 0^{+}$.
The space $M^{2}_{\lambda,\phi}$ is manifold with
corners and may be illustrated as in figure 
\ref{blow up space of resolvent}.

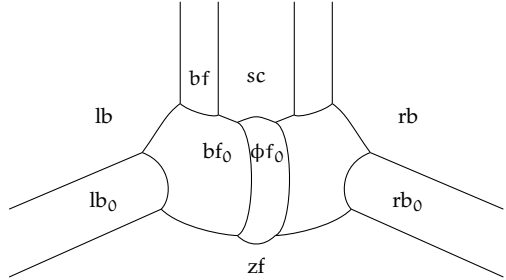
\begin{figure}[h]
\centering
\begin{tikzpicture}[scale =0.5]

\draw[-] (-2,2.3)--(-2,5);
\draw[-] (-1,2)--(-1,5);
\draw[-] (1,2)--(1,5);
\draw[-] (2,2.3)--(2,5);
\draw[-] (-3,1)--(-6.5,-0.5);
\draw[-] (3,1)--(6.5,-0.5);
\draw[-] (-2.5,-0.5)--(-6.5,-2.3);
\draw[-] (2.5,-0.5)--(6.5,-2.3);

\draw (-1,2)--(-0.5,1.8);
\draw (1,2)--(0.5,1.8);
\draw (-0.5,1.8).. controls (0,2) and (0,2) .. (0.5,1.8);

\draw (-1,2).. controls (-1.1,1.9) and (-1.8,2.1) .. (-2,2.3);
\draw (1,2).. controls (1.1,1.9) and (1.8,2.1) .. (2,2.3);

\draw (-2.5,-0.5).. controls (-2,-1) and (-0.8,-1.2) .. (-0.5,-1.2);
\draw (2.5,-0.5).. controls (2,-1) and (0.8,-1.2) .. (0.5,-1.2);
\draw (-0.5,-1.2).. controls (-0.2,-1.5) and (0.2,-1.5) .. (0.5,-1.2);

\draw (-3,1).. controls (-2.2,0.7) and (-2.2,-0.2) .. (-2.5,-0.5);
\draw (3,1).. controls (2.2,0.7) and (2.2,-0.2) .. (2.5,-0.5);
\draw (-2,2.3).. controls (-2.4,2) and (-2.5,1.7) .. (-3,1);
\draw (2,2.3).. controls (2.4,2) and (2.5,1.7) .. (3,1);

\draw (-0.5,-1.2).. controls (0,-1) and (0,1.6) .. (-0.5,1.8);
\draw (0.5,-1.2).. controls (1,-1) and (1,1.6) .. (0.5,1.8);

\node at (-4,-0.3) {\tiny{$\text{lb}_{0}$}};
\node at (4,-0.3) {\tiny{$\text{rb}_{0}$}};
\node at (-4,2) {\tiny{$\text{lb}$}};
\node at (4,2) {\tiny{$\text{rb}$}};
\node at (0,-2) {\tiny{\text{zf}}};
\node at (-1,1) {\tiny{$\text{bf}_{0}$}};
\node at (0.3,1) {\tiny{$\phi f_0$}};
\node at (0,3) {\tiny{$sc$}};
\node at (-1.5,3) {\tiny{$bf$}};

\end{tikzpicture}
 \caption{Blowup space $M^{2}_{\lambda , \phi}$}
  \label{blow up space of resolvent}
\end{figure}

The main result of \cite{reslaw gtv} , expresses low energy of 
resolvent of $\phi$ Hodge Laplacian in terms of
polyhomogeneous conormal distribution on $M^{2}_{\lambda, \phi}$. 
The result reads as follows,

\begin{Th}\cite{reslaw gtv} (Theorem 7.11) \label{loweneregy of resolvent}
Under the same assumptions as in \cite{reslaw gtv},
The resolvent $(\Delta_{\phi} + k^{2})^{-1}$ as $k\longrightarrow 0^+$ is an element of
the split calculus  (defined as in \cite{reslaw gtv}) where,
\begin{equation*}
\mathcal{E}_{sc} \geq 0, \, \mathcal{E}_{\phi f_0} \geq 0,\, \mathcal{E}_{bf_{0}} \geq -2, \, \mathcal{E}_{lb_0} , \mathcal{E}_{rb_0} > 0, \mathcal{E}_{zf} \geq -2.
\end{equation*}
The leading terms at $sc$, $\phi f_0$, $bf_0$ and $zf$ are 
of order $0,0,-2,-2$.
\end{Th}

 \begin{Rem}
 Theorem $\ref{loweneregy of resolvent}$ is obtained in parallel work \cite{kottke-rochon}.
 \end{Rem}

We apply (\ref{func res heat 2}) and determine the 
asymptotics of heat kernel in large time regime 
by integration over path $\Gamma$.
One parametrize $\Gamma$ as,\\
 For $\frac{\pi}{2}  < \varphi < \pi$
 and $ a\in \mathbb{R}^{+}$ fixed,
 $\Gamma$ splits in two paths $\Gamma_{1,a}$ 
 and $ \Gamma_{2,a}$
 where $\Gamma_{1,a}$ and $ \Gamma_{2,a}$
 are parametrised as,
  \begin{align} \label{gamma}
\Gamma_{1,a} &= ae^{i\theta}, \qquad  \nonumber
-\varphi \leq \theta \leq \varphi,\\ 
\Gamma_{2,a} &= re^{i\theta}, 
\qquad a \leq r < \infty.
\end{align}
Assume that $\mathcal{R}(\lambda, z, z')$ 
is the Schwartz kernel of resolvent $
(\Delta_{\phi} + e^{i\theta}\lambda)^{-1}, \,
\lambda > 0.$
Recall that in short time regime we used 
$\tau := t^{\frac{1}{2}}$. Consequently fix $ \omega = \tau^{-1}$
 and $ a = \omega^{2}$. $\omega \longrightarrow 0$ corresponds to $ t \longrightarrow \infty$ and
the resolvent at low energy i.e as $a \longrightarrow 0^+$.
We may split (\ref{func res heat 2}) into two parts
and write,\,(by suppressing $x, y, x',y')$),
\begin{align*}
H(\omega^{-2}, z , z') &=
 \frac{1}{2\pi i}\int_{\Gamma} e^{\frac{\lambda}
 {\omega^{2}}}\mathcal{R}
(\lambda, z ,z')d\lambda\\
& =  \frac{1}{2\pi i} \left( \int_{\Gamma_{1, \omega^{2}}}
 e^{\frac{\lambda}{\omega^{2}}}\mathcal{R}
 (\lambda, z, z')d\lambda + 
 \int_{\Gamma_{2,\omega^{2}}} e^{\frac{\lambda}
 {\omega^{2}}}
 \mathcal{R}(\lambda, z , z')d\lambda\right).
\end{align*}
\subsection*{Integration on $\Gamma_{1,\omega^{2}}$}
One may use, $
\lambda = \omega^{2}e^{i\theta}, 
\, -\varphi \leq \theta  \leq \varphi,$
and therefore,
\begin{equation} \label{path 1}
\frac{1}{2\pi i} \int_{\Gamma_{1,\omega^{2}}}
e^{\frac{\lambda}{\omega^{2}}}\mathcal{R}
(\lambda, z, z')d\lambda = 
 \frac{\omega^{2}}{2\pi}\int_{-\varphi}^{\varphi}
  e^{e^{i\theta}}\mathcal{R}
  (\theta, \omega, z , z')d\theta.
\end{equation}
Firstly we observed that the integral converges by the boundedness 
of integrand with respect to $\theta$.
By Theorem \ref{loweneregy of resolvent} for each  fixed 
$\theta$, the integrand in (\ref{path 1}) is polyhomogeneous 
 on
$M^{2}_{\omega,\phi}$ with conormal singularity 
at diagonal $\triangle_{\phi}$ and all
coefficients depend smoothly on 
$\theta \in [-\varphi, \varphi]$.
Consequently the integral (\ref{path 1}) is 
polyhomogeneous on $M^{2}_{\omega,\phi}$.\\
The index set of
$H(\omega^{-2}, z, z')$ is computed from index sets  
 $\mathcal{R}(\theta, \omega, z, z')$
and from $\omega^{2}$, i.e as $\omega$ is boundary defining function for faces $zf$, $bf_{0}$, $\phi f_0$, $lb_0$, $rb_0$,
 on the faces $\text{zf}$, 
$\text{bf}_{0}$, $\phi f_0$, $\text{lb}_{0}$, $\text{rb}_{0}$
we add $+2$ to the index sets of the resolvent Theorem
 \ref{loweneregy of resolvent}
and the index sets on other faces remain
 the same as those index sets of 
$\mathcal{R}(\theta, \omega, z, z')$.

\subsection*{Integration on $\Gamma_{2, \omega^{2}}$}

The path $\Gamma_{2,\omega^{2}}$ consists of 
two rays at  angels
 $\varphi$ and   $-\varphi$. The integration over $\Gamma_{2,\omega^{2}}$ 
 refers therefore
to the integration over these two rays.
We claim for convergence and polyhomogeneity
 of  integration along $\varphi$
ray. The argument for integration
at $-\varphi$ ray is the same.
Parametrising $\lambda = re^{i\varphi}$,
we express the integral over $\Gamma_{2,\omega^{2}}$ to be,

\begin{equation} \label{eq:1}
\int_{\Gamma_{2,\omega^{2}}}
e^{\frac{\lambda}{\omega^{2}}}\mathcal{R}
(\lambda,z,z')d\lambda =
\int_{\omega^{2}}^{\infty}
e^{\frac{r}{\omega^{2}}e^{i\varphi}}\mathcal{R}
(r, z, z^{'})dr ,
\end{equation}
And use change of variable $s = \sqrt r , dr = 2sds$
in (\ref{eq:1}) to write,

\begin{equation} \label{eq:2}
\int_{\Gamma_{2,\omega^{2}}} e^{\frac{\lambda}
{\omega^{2}}}\mathcal{R}(\lambda, z,z')d\lambda =
\int_{\omega}^{\infty} 2s
e^{(\cos\varphi)\frac{s^{2}}
{\omega^{2}}}e^{i(\sin \varphi)
\frac{s^{2}}{\omega^{2}}}\mathcal{R}
(s^{2}, z,z')ds.
\end{equation}
As $\frac{\pi}{2} < \varphi < \pi$ is fixed,
$cos(\varphi) < 0$ and therefore for
$ s\longrightarrow +\infty$ the $e^{(\cos\varphi)\frac{s^{2}}
{\omega^{2}}}$ and consequently  entire integrand decays 
to infinite order at $ s = \infty, $ 
which mean that the integral $(\ref{eq:2})$ converges.
In order to analyse the polyhomogeneity of
(\ref{eq:2}), we split
$\mathcal{R}(s^{2},z,z')$ into two pieces.
Namely we use
partition of unity and write, $\mathcal{R} = \mathcal{R}_{D} + \mathcal{R}_{C},$
where $\mathcal{R}_{D}$ is supported in a neighborhood of
$\triangle_{\phi}$ and $\mathcal{R}_{C}$ is supported away 
from $\triangle_{\phi}$.
Recall that the diagonal $\triangle_ {\phi}$
 intersects the $sc$, $\phi f_0$ and $zf$ faces of $M^{2}_{s,\phi}$.  Accordingly we may split $\mathcal{R}_{D}$ into three pieces
 and write, $ \mathcal{R}_{D} = \mathcal{R}_{1} +
 \mathcal{R}_{2} + \mathcal{R}_{3}
 $. $\mathcal{R}_{1}$ is supported away from $sc$ and near $zf$ and $\mathcal{R}_{2}$ is supported away from $zf$ and near $sc$ and
$\mathcal{R}_{3}$ is supported in interior of
 $bf_{0}$ near diagonal of $\phi f_0$.
We argue for the polyhomogeneity of integral (\ref{eq:2})
 for each of $\mathcal{R}_{i}, i = 1,2,3.$
  Compare to figure \ref{diagonal pic}.

 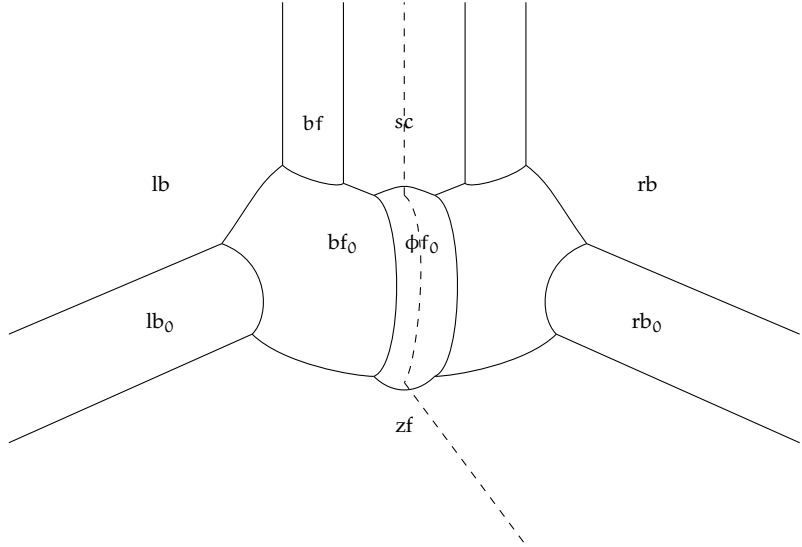
\begin{figure}[h]
\centering
\begin{tikzpicture}[scale =0.8]

\draw[-] (-2,2.3)--(-2,5);
\draw[-] (-1,2)--(-1,5);
\draw[-] (1,2)--(1,5);
\draw[-] (2,2.3)--(2,5);
\draw[-] (-3,1)--(-6.5,-0.5);
\draw[-] (3,1)--(6.5,-0.5);
\draw[-] (-2.5,-0.5)--(-6.5,-2.3);
\draw[-] (2.5,-0.5)--(6.5,-2.3);

\draw (-1,2)--(-0.5,1.8);
\draw (1,2)--(0.5,1.8);
\draw (-0.5,1.8).. controls (0,2) and (0,2) .. (0.5,1.8);

\draw (-1,2).. controls (-1.1,1.9) and (-1.8,2.1) .. (-2,2.3);
\draw (1,2).. controls (1.1,1.9) and (1.8,2.1) .. (2,2.3);

\draw (-2.5,-0.5).. controls (-2,-1) and (-0.8,-1.2) .. (-0.5,-1.2);
\draw (2.5,-0.5).. controls (2,-1) and (0.8,-1.2) .. (0.5,-1.2);
\draw (-0.5,-1.2).. controls (-0.2,-1.5) and (0.2,-1.5) .. (0.5,-1.2);

\draw (-3,1).. controls (-2.2,0.7) and (-2.2,-0.2) .. (-2.5,-0.5);
\draw (3,1).. controls (2.2,0.7) and (2.2,-0.2) .. (2.5,-0.5);
\draw (-2,2.3).. controls (-2.4,2) and (-2.5,1.7) .. (-3,1);
\draw (2,2.3).. controls (2.4,2) and (2.5,1.7) .. (3,1);

\draw (-0.5,-1.2).. controls (0,-1) and (0,1.6) .. (-0.5,1.8);
\draw (0.5,-1.2).. controls (1,-1) and (1,1.6) .. (0.5,1.8);

\node at (-4,-0.3) {\tiny{$\text{lb}_{0}$}};
\node at (4,-0.3) {\tiny{$\text{rb}_{0}$}};
\node at (-4,2) {\tiny{$\text{lb}$}};
\node at (4,2) {\tiny{$\text{rb}$}};
\node at (0,-2) {\tiny{\text{zf}}};
\node at (-1,1) {\tiny{$\text{bf}_{0}$}};
\node at (0.3,1) {\tiny{$\phi f_0$}};
\node at (0,3) {\tiny{$sc$}};
\node at (-1.5,3) {\tiny{$bf$}};
\draw[dashed] (0,1.8)--(0,5);
\draw[dashed] (0,-1.3)--(2,-4);
\draw[dashed] (0,1.8).. controls (0.5,1.4) and (0.2,-1.1)..(0,-1.3);

\end{tikzpicture}
 \caption{Diagonal of $M^{2}_{s,\phi}$}
 \label{diagonal pic}
  
\end{figure}

\subsubsection*{ Polyhomogeneity of $\mathcal{R}_{1}$}
On the support of $\mathcal{R}_{1}$ near $zf$,
 away from $x \neq 0$,
we may use projective coordinates, $
\hat{X} = \frac{x - x'}{x}, 
\hat{Y} = \frac{y - y'}{x},\hat{Z} = \frac{z-z'}{x},
 \mu = \frac{s}{x}.$
By definition of conormal singularity along 
diagonal $\triangle_{\phi}$, we may express,
\begin{equation} \label{eq:4}
\mathcal{R}_{1} \sim \int_{\mathbb{R}^{n}}
 e^{i(\hat{X},\hat{Y},\hat{Z})(\zeta_{1},\zeta_{2},\zeta_{3})}
  \sum_{j = 0}^{\infty} a_{j}
(\frac{s}{x},x,y,z,\frac{\zeta}
{\vert \zeta \vert}) \vert \zeta \vert^{2-j} d\zeta,
\end{equation}
where by $ \sim$ we mean that $\mathcal{R}_{1}$ may be expressed
locally by (\ref{eq:4})  on $M^{2}_{s,\phi}$ up to smooth function.
Thus we may absorb the smooth reminder into $\mathcal{R}_{C}$
and plug (\ref{eq:4}) into (\ref{eq:2}) to obtain,
 \begin{equation*}
\int_{\mathbb{R}^{n}} e^{i(\hat{X},\hat{Y},\hat{Z})
(\zeta_{1}, \zeta_{2}, \zeta_{3})} 
\sum_{j = 0}^{\infty}(\int_{\omega}^
{\infty} 2se^{(\frac{s}{\omega})^{2}}e^{i\varphi}
 a_{j}(\frac{s}{x}, x, y, z , \frac{\zeta}
 {\vert \zeta \vert})\vert \zeta \vert^{2-j}d\zeta ds.
 \end{equation*}
From definition of conormal singularity, one has that
the coefficients $a_{j}$ are polyhomogeneous in $x$
and $\frac{s}{x}$ with index set independent 
from $j$ and $a_{j}$ is smooth in  
$(y, z, \frac{\zeta}{\vert \zeta \vert})$. 
Observe that the pullback of $a_{j}$ via projection to, $
M^{3}_{b}(s, x, \omega) \times N_{y} \times N_{f}
\times \mathbb{S}^{n-1}_{\frac{\zeta}{\vert \zeta \vert}},$
is polyhomogeneous conormal with index set
independent of $j$. Therefore, $
2se^{(\frac{s}{\omega})^{2}}e^{i\varphi}a_{j}
(\frac{s}{x},x,y,z, \frac{\zeta}{\vert \zeta \vert}) $
is polyhomogeneous conormal on
$ M^{3}_{b}(s, \omega, x) \times N_{y} \times N_{f} \times \mathbb{S}^{n-1}_{\frac{\zeta}{\vert \zeta\vert}}$
and it admits cut off singularity at $\frac{s}{\omega} = 1$.
The integrand has infinite decay in $s$ due to the fact
that $\cos \varphi < 0$ and therefore integration
is well defined. Integration in $s$ corresponds
 to a projection map,
\begin{equation*}
M^{3}_{b}(s,\omega,x) \longrightarrow M^{2}_{b}(\omega,x),
\end{equation*}
which is b-fibration \cite{Mel-diff} by the following lemma from
 \cite{sc conf}.

\begin{Lemma}
If $m < n$ each of the projections off $n-m$ factors of $X$, $
\pi: X^{n} \longrightarrow X^{m},$
fixes a unique b-stretched projection $\pi_{b}$
 giving a commutative diagram,

\begin{equation}
\begin{tikzcd}
X^{n}_{b}\arrow[r]\arrow[d]& X^{m}_{b}\arrow[d] \\X^{n}
\arrow[r]& X^{m}
\end{tikzcd}
\end{equation}
and furthermore $\pi_{b}$ is b-fibration.
\end{Lemma}
Extension to 
$N_{y} \times N_{z} \times 
\mathbb{S}^{n-1}_{\frac{\zeta}{\vert \zeta \vert}}$ by
$(y, z, \frac{\zeta}{\vert \zeta \vert})$
does not change the b-fibration property of this map
and consequently by pushforward theorem of Melrose
 \cite{MelATP}, $
\int_{\omega}^{\infty}
2se^{(\frac{s}{\omega})^{2}}e^{i\varphi}a_{j}
(\frac{s}{x},x,y,z, \frac{\zeta}{\vert \zeta \vert})ds,$
is polyhomogeneous conormal on,\\ $
 M^{2}_{b}(\omega,x) \times N_{y} \times N_{z}
  \times \mathbb{S}^{n-1}_{\frac{\zeta}{\vert \zeta \vert}},$
with index set independent of $j$ and consequently on
 $M^{2}_{\omega,\phi}.$

 \subsubsection*{Polyhomogeneity of $\mathcal{R}_{2}$}

We need to prove the polyhomogeneity of the integration
 of kernel $\mathcal{R}_{2}$ along diagonal
 $\triangle_{\phi} \cap sc$ of $M^{2}_{s,\phi}$. \\
 $\mathcal{R}_{2}$ is supported near $sc$
 face of $M^{2}_{s,\phi}$. The coordinates we employ
are in the form, $
 X = s(\frac{1}{x} - \frac{1}{x'}), 
 Y = \frac{s}{x}(y - y'), \mu = \frac{x}{s}, s, Z = \frac{s}{x}(z-z').$
The definition of conormality implies that,

\begin{equation} \label{eq:5}
\mathcal{R}_{2} \sim \int_{\mathbb{R}^{n}} 
e^{i(X, Y, Z)(\zeta_{1},\zeta_{2},\zeta_{3})}
\sum_{j = 0}^{\infty} b_{j}(\frac{x}{s},s,y,z,
\frac{\zeta}{\vert \zeta \vert}) \vert \zeta
 \vert^{2-j}d \zeta.
\end{equation}
Where $b_{j}$ are polyhomogeneous conormal in
$\frac{x}{s}$ and $s$ with index sets independent
of $j$ and $b_{j}$ smoothly depend on 
$(y,z, \frac{\zeta}{\vert \zeta \vert})$. 
We plug (\ref{eq:5}) into (\ref{eq:2}) and analyse
the integral
by arguing in two different regimes which correspond
to $\omega > \frac{x}{2}$ and $\omega < \frac{x}{2}.$\\
First we assume that $\omega > \frac{x}{2}$
and we introduce projective coordinates, on sc face of
 $M^{2}_{\omega,\phi}$,
$\bar{X} = \omega(\frac{1}{x} - \frac{1}{x'}),
\bar{Y} = \frac{\omega}{x}(y - y'), \bar{Z} = \frac{\omega}{x}(z - z'),
 \bar{\lambda} = \frac{x}{\omega}.$
We need expression of conormal singularity at,
 $\bar{X} = \bar{Y} = \bar{Z} =0,$
in this region. We note that, $
(X, Y, Z) = (\frac{s}{\omega}\bar{X}, \frac{s}{\omega}\bar{Y},\frac{s}{\omega}\bar{Z})$
and take new variable $\eta$ which is defined as, $
\eta = (\eta_{1}, \eta_{2}, \eta_{3}) = (\frac{s}{\omega}\zeta_{1},\frac{s}{\omega}\zeta_{2}, \frac{s}{\omega}\zeta_{3}).$
We  plug (\ref{eq:5}) in (\ref{eq:2}) and employ new variables
$\bar{X}, \bar{Y}, \bar{Z} $. Interchanging
the integral and the sum to obtain,
\begin{equation*}
\int_{\mathbb{R}^{n}} \int_{\omega}^{\infty}
e^{i(\bar{X}, \bar{Y},\bar{Z})(\eta_{1},\eta_{2},\eta_{3})}
\sum_{j = 0}^{\infty}
2s e^{(\frac{s}{\omega})^{2}}e^{i\varphi} b_{j}
(\frac{x}{s},s,y,z,
\frac{\eta}{\vert \eta \vert})
(\frac{s}{\omega})^{j-2-(n)}\vert \eta \vert^{2-j}
 d s d \eta.
\end{equation*}
We argue that, $
\int_{\omega}^{\infty} 2s e^{(\frac{s}{\omega})^{2}}
e^{i\varphi}b_{j}(\frac{x}{s},s,y,z,\frac{\eta}
{\vert \eta \vert})(\frac{s}{\omega})^{j-2-(n)}ds,$
is polyhomogeneous conormal in 
$\frac{x}{\omega}, \omega$ independent of $j$.
 The argument is similar to the last step as the
integrand decays to infinite order at $ s = \infty$ 
and is polyhomogeneous conormal in $\frac{x}{s}, s$
independent of $j$. Integration with respect to $s$
yields the polyhomogeneity on $X^{2}_{b}(x, \omega)
\times N_{y} \times N_{f} \times \mathbb{S}^{n-1}_
{\frac{\eta}{\vert \eta \vert}}$
by Melrose pushforward Theorem. \\ 
The second region corresponds to $\omega < \frac{x}{2}$.
Here one use coordinates, $
\hat{X} =  \frac{x-x'}{x}, \hat{Y} 
= y - y', \hat{Z} = z-z' , \frac{s}{x}, x,$
and one expects that conormal singularities arise on, $
\hat{X} = \hat{Y} = \hat{Z} = 0.$
One notes that $( X , Y, Z) = (\frac{s}{x})(\hat{X}, \hat{Y},\hat{Z})$. One introduce variables $(\zeta^{'}_{1},\zeta{'}_{2},\zeta^{'}_{3}) = \frac{s}{x}(\zeta_1,\zeta_2,\zeta_3)$ and plug these variables in (\ref{eq:5}) to obtain,
\begin{equation*}
\int_{\mathbb{R}^{n}} e^{i(\hat{X}, \hat{Y}, \hat{Z})
(\zeta'_{1},\zeta'_{2},\zeta'_{3})}
 \sum_{j =0}^{\infty}(\int_{\omega}^{\infty}
  2s e^{(\frac{s}{\omega})^{2}e^{i\varphi}}
  b_{j}(\frac{x}{s},s,y,z,\frac{\zeta'}
  {\vert \zeta' \vert})(\frac{s}{x})^{j-2-n}
  ds\vert \zeta' \vert^{2-j})d\zeta',
 \end{equation*}
The j-coefficient may be written as,

\begin{equation} \label{eq:6}
(\frac{\omega}{x})^{j-2-n}\int_{\omega}^
{\infty}2s e^{(\frac{s}{\omega})^{2}}e^
{i\varphi}b_{j}(\frac{x}{s},s,y,z,
\frac{\zeta'}{\vert \zeta' \vert})(\frac{s}{\omega})
^{j-2-n}ds,
\end{equation}
which is $(\frac{\omega}{x})^{j-2-n}$ times 
(\ref{eq:4}). The 
$(\frac{\omega}{x})^{j-2-n}$ is polyhomogeneous
conormal on 
$M^{2}_{b}(\omega,x)$ and consequently (\ref{eq:6}) is 
polyhomogeneous conormal on $M^{2}_{b}(\omega,x)$ 
for each $j$. As $\omega < \frac{x}{2}$ by increasing j
 the order of polyhomogeneity increase and hence 
 all coefficients are polyhomogeneous conormal
  on $M^{2}_{b}(\omega,x)$ with respect
to the index set of the $j = 0$ coefficient.\\
We conclude that the integration on $\mathcal{R}_{2}$
yields to the polyhomogeneity of (\ref{eq:5}) on
$M^{2}_{\omega,\phi}$.

\subsubsection*{Polyhomogeneity of $\mathcal{R}_{3}$}

The polyhomogeneity  integration 
(\ref{eq:1}) with respect to  $\mathcal{R}_{3}$
follows from the fact that
$\mathcal{R}_{3}$ is supported in
compact region namely on diagonal of $\phi f_0$,
and by conormality (\ref{conormal def})
we may express by adequate local coordinates on $\phi f_0$,

\begin{equation} \label{eq:7}
\mathcal{R}_{3} \sim \int_{\mathbb{R}^{n}}
 e^{i(z-z')(\eta)}\sum_{j=0}^{\infty}
 c_{j}(s,z,\frac{\eta}{\vert \eta \vert})\vert 
 \eta \vert^{2-j}d\eta,
\end{equation}
where $c_{j}$ are polyhomogeneous conormal at
 $s = 0$ and $s = \infty$, with index sets
 independent 
 of $j$ at $s = 0$ and $s = \infty$. 
 We plug (\ref{eq:7}) into
 (\ref{eq:2}) and we get,
 
\begin{equation} \label{eq: 8}
 \int_{\mathbb{R}^{n}}e^{i(z-z')\eta}
 \sum_{j=0}^{\infty}(\int_{\omega}^{\infty}2s e^{(\frac{s}{\omega})^{2}}e^{i\varphi}c_{j}(s,z,\frac{\eta}{\vert \eta \vert})ds\vert \eta \vert^{2-j}d\eta.
\end{equation}
Note that  the j-th coefficient, $
 \int_{\omega}^{\infty}2s e^{(\frac{s}{\omega})^{2}}
 e^{i\varphi}c_{j}(s,z,\frac{\eta}{\vert \eta \vert})ds,$
is polyhomogeneous conormal on $M^{2}_{b}(s,\omega)$
with index sets independent of 
$j, z, \frac{\eta}{\vert \eta \vert}$
with infinite decay at $ s = \infty$. Consequently (\ref{eq: 8})
is polyhomogeneous on $M^{2}_{\omega,\phi}$.

\subsubsection*{Polyhomogeneity of $\mathcal{R}_{\mathcal{C}}$}

Now we argue the polyhomogeneity of $\mathcal{R}_{\mathcal{C}}$
term. Explicitly one may write, by suppressing $(x,y,z)$ in $z$
 and $(x',y',z')$ in $z'$,\\ $
\int_{\omega}^{\infty} 2se^{(\frac{s}{\omega})^{2}}
e^{i\varphi}\mathcal{R}_{s}(s^{2}, z, z')ds.$
By assumption 
$\mathcal{R}_{s}(s^{2},x,y, z,x', y', z')$
is polyhomogeneous conormal on $ M^{2}_{s,\phi}$
and smooth across the diagonal $\triangle_{s,\phi}$.
We apply the following lemma parallel to
 \cite{She-stab} (Lemma 10) to the integral kernel,
\begin{equation} \label{eq:9}
\mathcal{R}_{s}(s^{2},x,y,z, x', y', z').
\end{equation}
to conclude that (\ref{eq:9}) is polyhomogeneous 
on $M^{2}_{\omega,\phi}$.

\begin{Lemma} \label{poly smooth}
Assume $T(k,z,z')$ be a function which is
polyhomogeneous 
conormal on $M^{2}_{k,\phi}$ 
and smooth 
in the interior and decaying to infinite order at lb,
rb, and bf. Then,

\begin{equation} \label{eq:10}
\int_{\omega}^{\infty} 2ke^{(\frac{k}{\omega})^{2}}
e^{i\varphi}T(k,z,z')dk,
\end{equation}
is polyhomogeneous on $M^{2}_{\omega,\phi}$ for
$\omega$ bounded from below.
\end{Lemma}
\begin{proof}
In the following proof we use explicit local coordinates
 in each region of $M^{2}_{k,\phi}$ in order to plug the polyhomogeneity expression of $T(k,z,z')$ and evaluate
  directly the integral
 (\ref{eq:10}) to show that the resulting function
is polyhomogeneous on $M^{2}_{\omega,\phi}$.

\subsubsection*{Near $bf\cap bf_{0} \cap lb$}
 We may use coordinates, $(\zeta =\frac{x}{k},s = \frac{x'}{x},k)$ 
 in this region and the polyhomogeneity of 
 $T(k,x,x')$ with respect to these coordinates becomes $T \sim \sum a_{ijl}\zeta^{i}s^{j}k^{l}$. We plug this expression 
 into (\ref{eq:10}) and use change of variable
  $\frac{k}{\omega} = t$ and obtain,
 \begin{align*}
 &\int_{\omega}^{\infty} 2k e^{(\frac{k}{\omega})^{2}}e^{i\varphi}T(k,x,x')dk =
 \sum\int_{\omega}^{\infty} 2ke^{(\frac{k}{\omega})^{2}}e^{i\varphi}a_{ijk}k^{l-i}x^{i}s^{j}dk = \\
 &\sum a'_{ijl}2 x^{i}s^{j}\int_{1}^{\infty} \omega t e^{-t^{2}}(\omega t)^{l-i}\omega dt = \sum 2a'_{ijl}x^{i}s^{j}\omega^{l-i+2}\int_{1}^{\infty}
 e^{-t^{2}}t^{l-i+1}dt\\
 & \sim \sum a'_{ijl} (\frac{x}{\omega})^{i}
 s^{j}\omega^{l+2},
\end{align*}
i.e the integral is polyhomogeneous with respect to coordinates
$(\frac{x}{\omega}, s, \omega)$ that corresponds to region
 $bf \cap bf_{0} \cap lb$  of $M^{2}_{\phi,\omega}$. 
As the argument is symmetric near $bf \cap bf_{0} \cap rb$
 we leave the proof.

\subsubsection*{Near $lb \cap lb_{0} \cap bf_{0}$.}
In this region, we may use coordinates
 $(x , s' = \frac{x'}{x}, \kappa = \frac{k}{x})$. Plugging the definition of
 polyhomogeneity of $T \sim \sum a_{ijl} x^{i}s'^{j}\kappa^{l}$
 into integral $(\ref{eq:10})$ and using change of variable $\frac{k}{\omega} = t$ yields to,
  \begin{align*}
 &\int_{\omega}^{\infty} 2k e^{(\frac{k}{\omega})^{2}}e^{i\varphi}T(k,x,x')dk =
 \sum\int_{\omega}^{\infty} 2ke^{(\frac{k}{\omega})^{2}}e^{i\varphi}a_{ijl}x^{i}s'^{j}\kappa^{l}dk = \\
 &\sum a'_{ijl}2 x^{i}s'^{j}\int_{1}^{\infty} \omega t e^{-t^{2}}e^{i\varphi}(\omega t)^{l}\omega dt = \sum 2a'_{ijl}x^{i-l}s^{j}\omega^{l}\int_{1}^{\infty}
 e^{-t^{2}}t^{l}dt\\
 & \sim \sum a'_{ijl} x^{i}s'^{j}(\frac{\omega}{x})^{l},
\end{align*}
which means the integral is polyhomogeneous with respect to coordinates 
$(x,\frac{x'}{x}, s, \frac{\omega}{x})$.
That means the polyhomogeneity 
on the region $lb \cap lb_{0} \cap bf_{0}$. 
The similar argument shows also the polyhomogeneity  of integral on 
 the region $rb \cap rb_{0} \cap bf_{0}$. 
 
\subsubsection*{Near $sc \cap bf_{0}$.} We use the coordinates
$(S = \frac{k(x-x')}{x'^{2}}, S' = \frac{x'}{k}, k)$ and the polyhomogeneity of $T(k,x,x')$ with respect to these coordinates
$ T \sim \sum a_{ijl}S^{i}S'^{j}k^{l}$ in the integral 
(\ref{eq:10}), and use change of variable $\frac{k}{\omega} = t$,
 \begin{align*}
 &\int_{\omega}^{\infty} 2k e^{(\frac{k}{\omega})^{2}}e^{i\varphi}T(k,x,x')dk =
 \sum\int_{\omega}^{\infty} 2ke^{(\frac{k}{\omega})^{2}}e^{i\varphi}a_{ijl}S^{i}S'^{j}k^{l}dk = \\
 &\sum a'_{ijl}2(\frac{x-x'}{x'^{2}})^{i}x'^{j}\int_{1}^{\infty} \omega^{i-j+1} t^{i-j+1} e^{-t^{2}}e^{i\varphi}(\omega t)^{i-j+1+l}
 \omega dt =\\
  &\sum 2a'_{ijl}S^{i}S'^{j}\omega^{l}\int_{1}^{\infty}
 e^{-t^{2}}t^{l+i-j+1}dt
  \sim \sum a'_{ijl} S^{i}S'^{j}\omega^{l},
\end{align*}
which yields to polyhomogeneity of integral in the region $sc \cap bf_{0}$ of $M^{2}_{\omega,\phi}$. Similar argument shows the polyhomogeneity near the face $\phi f_0$ of $M^{2}_{\omega,\phi}$ as well.
\end{proof}
In the following theorem, we summarize
the polyhomogeneity of the heat kernel in long time regime.

\begin{Th}\label{heat kernel asym}
The heat kernel which is given by (\ref{func res heat 2}),
 i.e
$$H^{M}(t,x,x') = \frac{1}{2\pi i} \int_{\Gamma} e^{t\lambda}(\Delta_{\phi} + \lambda)^{-1} d\lambda,$$
is polyhomogeneous conormal at $ t = \omega^{-\frac{1}{2}}$
at $\omega \longrightarrow 0$ on $M^{2}_{\omega,\phi}$ 
with index sets given in terms of index sets of 
resolvent $(\Delta_{\phi} + \lambda)^{-1}$
at low energy level. More explicitly the
asymptotics of heat kernel
in long time regime are of leading order $0$ at
$sc$ face and of order $0$ at $zf$ and $bf_{0}$ faces. 
More over the leading order at the face $\phi f_{0}$ 
is $2$.
In long time regime the heat kernel vanishes to
infinite order at lb, rb, and bf faces of $M^{2}_{\omega,\phi}$.
The explicit index sets are as follows,
$$\mathcal{E}_{sc} \geq 0,\, \mathcal{E}_{\phi f_0} \geq 2,\, \mathcal{E}_{bf_0} \geq 0, \, \mathcal{E}_{lb_0}, \mathcal{E}_{rb_0} >0, \mathcal{E}_{zf} \geq 0.$$
 \end{Th}

\section{Analytic torsion}\label{sec4}

For $(M,g)$  compact oriented $C^{\infty}$ Riemannian manifold 
of dimension $dim(M) = n$ and $\Delta^{q}_{g}$ Hodge Laplacian acting on $\Omega^{q}(M)$, let
\begin{equation*}
0 \leq \lambda_{1} \leq \lambda_{2} \leq \cdots \leq \lambda_{n} \leq \longrightarrow +\infty,
\end{equation*}
be the sequence of eigenvalues of $\Delta^{q}_{g}$. The zeta function
$\Delta^{q}$ is then defined by,
\begin{equation*}
\zeta_{q}(s) = \sum_{\lambda_{i} > 0} \lambda_{i}^{-s},
\end{equation*}
and it turns out that the zeta function converges in
 the half plane
$\text{Re}(s) > \frac{n}{2}.$
One can explicitly express the zeta function by
 integration of heat kernel along diagonal.
  Namely for $ \text{Re}(s) > \frac{n}{2}$,
\begin{equation*}
\zeta_{M,q}(s) = \frac{1}{\Gamma(s)}\int_{0}^{\infty}Tr(H^{M}(t))t^{s-1}dt.
\end{equation*}
$\zeta_{M,q}(s)$ admits holomorphic continuation
 to $\mathbb{C}$ and especially with regular value at $0$ and 
 one may define the determinant
of Hodge Laplacian as\\
 $ exp(-\frac{d}{ds}\zeta_{M,q}(s)\vert_{s=0})$
  and analytic torsion is defined to be,
\begin{equation*}
Log T_M = \frac{1}{2} \sum_{q=0}^{n} (-1)^{q}q \frac{d}{ds}\zeta_{M,q}(s)\restriction_{s =0}.
\end{equation*}
Consider now  $\phi$ metric, $g_{\phi} =
 \frac{dx^{2}}{x^{4}} + \frac{\phi^{*}g_{B}}{x^{2}} + g_{F}+ O(x).$
One observe that the integration along diagonal of heat kernel
 diverges as $x\longrightarrow 0$. The way to overcome to this problem is to define
renormalization. 
In this section, We describe two ways
of renormalization. Renormalization together with  structure 
Theorems of heat kernel \ref{short time}, \ref{heat kernel asym} 
give rise to the definition of renormalized heat trace
and renormalized zeta function. 
We show that the renormalized zeta function admits meromorphic continuation to $\mathbb{C}$. We define 
determinant of Laplacian and renormalized zeta function
in the setup of fibred boundary manifold $(M,g_{\phi})$.

\subsection{Renormalized heat trace}
 We follow \cite{renor} in order to define two ways of renormalization.
 \subsubsection*{Riesz renormalization}
The first method of renormalization
is called the renormalization in the sense of
Riesz.
Assume $M$ is a manifold with boundary and $x$ is
a boundary defining function.
Assume further that $f$ admits an asymptotic 
expansion in terms of $x$ and $logx$ i.e,

\begin{equation} \label{exp f}
f \sim \sum a_{s,p}x^{s}(\text{log(x)})^{p}
\end{equation}
Assume that $\mu$ is smooth non-vanishing density
on $M$. The expansion of $f$ (\ref{exp f}) implies the
meromorphic 
continuation of the function, $
z \in \mathbb{C} \mapsto \int_{M} x^{z}f d\mu,$ for $\text{Re}(z) > C$.
Now the Riesz renormalized integral of $f$ is defined to be,
 $^R\int_{M} f d\mu = \underset{z = 0}{\text{FP}}
 \int_{M} x^{z}f d\mu.$

\subsubsection*{Hadamard renormalization}
The second method is due to Hadamard and is referred in
\cite{MelATP} as b-integral. Assume $M$ is manifold
 with boundary and $x$ is boundary defining function.
One may define,

\begin{equation}\label{int eps}
\epsilon \mapsto \int_{x \geq \epsilon} f d\mu,
\end{equation}
and show that (\ref{int eps}) has asymptotic expansion as
$\epsilon \longrightarrow 0$ when $f$ admits expansion 
(\ref{exp f}).
 We define then $
^h\int_{M}fd\mu = \underset{\epsilon = 0}{\text{FP}}
\int_{x \geq \epsilon}fd\mu,$
where by finite part we mean taking out the divergent part.

\subsubsection*{Heat kernel renormalization in the sense of Hadamard}

\begin{defn} \label{renormalized heat trace}
Assume $\overline{M}$ is a fibred boundary manifold
 and assume further that $H^{M}_{\phi}(x,y,z,x',y',z')$ 
is the heat kernel.
 One defines the renormalized 
heat trace to be,
\begin{equation*}
^R\text{Tr}(H^{M}_{\phi}(t)) = 
^h \int_{M} H^{M}_{\phi}(t,x,y,z,x,y,z) dvol_{\phi}(x,y,z).
\end{equation*}
\end{defn}
Equivalently one may take a sharp cutoff function $\chi(r)$ 
which is supported on
$[0,1]$ and is equal to $1$ for $ 0 \leq r \leq 1$.
For fix
 $\epsilon >0$ consider $\chi(\frac{\epsilon}{x})$.
The renormalized heat trace is defined as,
\begin{equation}\label{renor heat trace}
^R\text{Tr}(H^{M}_{\phi}(t)) = 
\underset{\epsilon = 0}{\text{FP}}\int_{M} \chi(\frac{\epsilon}{x}) H^{M}_{\phi}(t,x,y,z,x,y,z) dvol_{\phi}(x,y,z).
\end{equation}

it needs to be justified that 
(\ref{renor heat trace}) make sense.
We demonstrate to that end  that the integrand in
(\ref{renor heat trace}) is polyhomogeneous in
$\epsilon$ for fixed $t$.
  
\begin{Lemma} \label{heat kernel diag poly}
Assume $(\overline{M}, g_{\phi})$ is a fibred boundary
manifold and $\Delta_{\phi}^q$ is the Hodge Laplacian acting
on the space of $q$ forms.
The corresponding heat kernel is polyhomogeneous
conormal along diagonal $(x,y,z; x,y,z)$ in both short time
regime and long time regime.
Precisely,

\begin{itemize}

\item For $t$ bounded above, $
H^{M}_{\phi}(t,x,y,z,x,y,z),$
is polyhomogeneous conormal in $(\tau = \sqrt{t}, x)$ 
with smooth dependence on $y$ and $z$.
\item For $t$ bounded below and 
$\omega = t^{-\frac{1}{2}}$, $
H^{M}_{\phi}(t,x,y,z,x,y,z)$
is polyhomogeneous conormal as a function of 
$(\omega,x)$ on  $b$-space
 $(\mathbb{R}_+(\omega)\times M)_{b}$
 with smooth dependence on $y, z$.
\end{itemize}
\end{Lemma}

\begin{proof}
The lemma follows from the structure theorems
of heat kernel in last section i.e  from theorems 
\ref{short time} and \ref{heat kernel asym}.

\end{proof}

\begin{Th} \label{poly renorm heat trace}
The renormalized heat trace defined by 
(\ref{renor heat trace}) is well defined 
the integrand admits
expansion in $\epsilon$ for each 
fixed $t$. Moreover the finite part at 
$\epsilon = 0$ has
leading asymptotics at $t = 0$ 
and $ t =\infty$. More precisely for coefficients $a_{j}$
 and $b_{jl}$, 
\begin{align}
&^R Tr(H_{\phi}^{M}(t)) \sim_{t\longrightarrow 0} \sum_{j \geq 0} a_{j}t^{\frac{-m+j}{2}}, \label{short time asymp trace}\\
&^R Tr(H_{\phi}^{M}(t)) \sim_{t\longrightarrow \infty} \sum_{j \geq 0} \sum_{l=0}^{p_{j}}b_{jl}t^{-\frac{j}{2}}log^{l}(t). \label{longtime asymp trace}
\end{align}

\end{Th}
\begin{proof}
Consider the integral,

\begin{equation}\label{eq: ren}
\int_{M} \chi(\frac{\epsilon}{x}) H^{M}_{\phi}(t,x,y,z,x,y,z) dvol_{\phi}(x,y,z).
\end{equation}

For fixed t and $\epsilon$ the integrand is polyhomogeneous on
$(\mathbb{R}\times M_{x})_{b} \times \mathbb{R}_{t}^{+}$ for short time and $(\mathbb{R}_{\epsilon}\times M_{x} \times \mathbb{R}_{t}^{+})_{b}$ in long time. The projection $\pi_{x}$ lift to b-fibrations,
\begin{align*}
&\Pi_{x,b}: (\mathbb{R}\times M_{x})_{b} \times \mathbb{R}_{t}^{+} \longrightarrow M_{x} \times \mathbb{R}_{t}^{+},\\
&\Pi_{x,b2}:(\mathbb{R}_{\epsilon}\times M_{x} \times \mathbb{R}_{t}^{+})_{b} \longrightarrow ( M_{x} \times \mathbb{R}_{t}^{+})_{b}.
\end{align*}
Integration in $x$ corresponds to pushforward under $\Pi_{x,b}$ (short time) and $\Pi_{x,b2}$ (long time), the polyhomogeneity 
of integral with respect to $\epsilon$ follows from Melrose push forward theorem.

\subsubsection*{Asymptotic in (\ref{short time asymp trace})} 
By  Lemma \ref{heat kernel diag poly} one can express the diagonal of heat kernel in short time regime, for $ N \in \mathbb{N}$ as,
\begin{equation}\label{eqn: 16}
 Tr_{x}H_{\phi}^{M}(t,x) = \sum_{j=0}^{N} x^{j}a_{j}(t) + H_{N}(x,t),
\end{equation}
where $a_{j}(t) \sim_{t\longrightarrow 0} \sum_{k\geq 0}
 \tau^{-m+k}$ and $H_{N}(x,t) = O(\tau^{-m}x^{N+1}).$ As $\epsilon \longrightarrow 0$ we may Plug
 (\ref{eqn: 16}) into (\ref{eq: ren})
  and evaluate the integral to obtain (\ref{short time asymp trace}).
 \subsubsection*{Asymptotic in (\ref{longtime asymp trace})}
 The diagonal of heat kernel in long time regime is of leading order $0$ at $sc$ and $2$ at $\phi f_{0}$ and $0$ at $zf$ face of $M^{2}_{\omega,\phi}$. As $\epsilon \longrightarrow 0$ one may express
 explicit diagonal of heat kernel i.e Lemma \ref{heat kernel diag poly} at $\omega^{-\frac{1}{2}} = t = \infty$. On $\phi f_{0}$ face 
 one obtain,
 $$Tr_{x} H^{M}_{\phi}(x,\omega) \sim \sum h_{j}(x,\omega),$$
  for $h_{i}(x,\omega)$ homogeneous of order $2$. Similarly on
 $zf$ face we have,
  $$Tr_{x} H^{M}_{\phi}(x,\omega) \sim \sum a_{j}x^{j}.$$
   Plugging these expressions into (\ref{eq: ren}) 
 with respect to $\phi$-volume form $dvol_{\phi} = x^{-b-2}dxdydz$ we obtain
 (\ref{longtime asymp trace}).
 
\end{proof}
Consider formally, 
\begin{equation}\label{ren zeta}
 \frac{1}{\Gamma(s)}\int_{0}^{\infty}\, 
^R\text{Tr}(H^{k}_{M,\phi})(t)t^{s-1}dt.
\end{equation}
Apriori (\ref{ren zeta}) is not defined for any $s\in \mathbb{C}$.
By breaking (\ref{ren zeta}) at some constant $c$ we may express $(\ref{ren zeta})$ as sum of two integrals.
By $(\ref{short time asymp trace})$ the integral,
\begin{equation}\label{ren zeta 0}
_0^R \zeta_{M,\phi}^{k}(s):= \frac{1}{\Gamma(s)}\int_{0}^{c}\, 
^R\text{Tr}(H^{k}_{M,\phi})(t)t^{s-1}dt,
\end{equation}
is defined for $Re(s) > \frac{n}{2}$. Each summand can
directly be evaluated to show that $(\ref{ren zeta 0})$ admits meromorphic extension to complex plane $\mathbb{C}$. 
The second integral is denoted as $_\infty ^R \zeta_{M,\phi}^{k}(s)$,
\begin{equation}\label{ren zeta inf}
_\infty ^R\zeta_{M,\phi}^{k}(s):= \frac{1}{\Gamma(s)}\int_{c}^{\infty}\, 
^R\text{Tr}(H^{k}_{M,\phi})(t)t^{s-1}dt.
\end{equation}
One can apply $(\ref{longtime asymp trace})$ and show that the integral 
converges for $Re(s) <0$ and by evaluating directly $(\ref{ren zeta inf})$ the meromorphic extension to complex plane $\mathbb{C}$ follows.

\begin{defn} Assume $(\overline{M},g_{\phi})$ is fibred boundary 
manifold and
 $\Delta_{\phi}^{k}$
is the Hodge Laplacian
acting on the space of $k$ forms.  One defines,

\begin{itemize}

\item The renormalized zeta function on $(\overline{M},g_{\phi})$
 at degree
 $k$, denoted as
$^R\zeta_{M,\phi}^{k}(s)$ is defined to be,
\begin{equation}
^R \zeta_{M,\phi}^{k}(s)  := _0^R \zeta_{M,\phi}^{k}(s)
 + _\infty ^R\zeta_{M,\phi}^{k}(s)
\end{equation}

\item The renormalized determinant of the Laplacian on $\overline{M}$
 is denoted as $-\frac{d}{ds}^R\zeta_{M,\phi}(s)\vert_{s=0}$,
where $\frac{d}{ds}^R\zeta_{M,\phi}(s)\vert_{s=0}$
 is the coefficient of $s$ in the Laurent series 
 for $^R \zeta_{M,\phi}(s)$ at $s = 0$.
\end{itemize}
\end{defn}

One may define renormalized analytic torsion on fibred boundary manifold as,
\begin{defn} \label{analytic torsion phii}
For $(\overline{M},g_{\phi})$ fibred boundary manifold, denote
 $ \Delta_{\phi}^{k}$ 
to be Hodge Laplacian acting on the space of
$k$ forms. One may define the renormalized
analytic torsion by,
\begin{equation}
\text{Log}\,^R T_{M,g_{\phi}} := \frac{1}{2}\sum_{q=0}^{n}(-1)^{q}q\,
\frac{d}{ds}\, ^R\zeta^{q}_{M,\phi}(s)\vert_{s=0}.
\end{equation}
\end{defn}
We conclude the discussion of this section pointing out 
that the renormalized analytic torsion as defined in 
$\ref{analytic torsion phii}$ may be studied further and one
 may ask for the statement similar to Cheeger M\"uller
  Theorem in the fibred boundary manifold setup.
\begin{Prob}\textit{(Cheeger M\"uller in the set up of $\phi$ manifolds)}
\begin{itemize}
\item[1.] Define the toplogical torsion in the set up of manifold with fibred boundary. Is this trivial extension from closed manifolds?
\item[2.] Prove  Cheeger M\"uller type statement.
\end{itemize}
\end{Prob}

\end{document}